\documentclass[12pt,reqno]{amsart}
\usepackage{amsmath} % 提供数学符号和排版命令
\usepackage{amssymb} % 提供数学符号
\usepackage{amsthm}  % 提供定理等环境
\usepackage{graphicx}% 支持插入图片
\usepackage{geometry}% 设置页面边距
\usepackage{babel}   % 支持多语言，这里设置为英语
\usepackage{lipsum}  % 用于生成示例文本
\usepackage{tikz}
\usetikzlibrary{decorations.pathreplacing, calc, arrows.meta}
\usepackage{array}
\usepackage[numbers]{natbib} % Combined and placed before hyperref
\usepackage{float}           % Fixed: Removed hidden non-breaking space
\usepackage{xtab}            % Fixed: Removed hidden non-breaking space
\usepackage{enumerate}
\usepackage{color}
\usepackage{epic}
\usepackage[pagebackref,colorlinks=true]{hyperref} % Consolidated hyperref options

\usepackage{microtype}
\usepackage{mathtools}

\newtheorem{Theorem}{Theorem}[section]
\newtheorem{Lemma}[Theorem]{Lemma}

\numberwithin{equation}{section}
\newtheorem{Example}{Example}[section]
\newtheorem{remark}[Theorem]{Remark}

\newcommand{\len}{\mathrm{len}}
\newcommand{\f}{\mathrm{f}}

\def\bZ{\mathbb{Z}}
\def\mo{\mathrm{mo}}
\def\me{\mathrm{me}}

\newcommand{\tpmod}[1]{\mkern 12mu(\textup{mod}\mkern 6mu #1)}

\newcommand{\bisq}{%
  \begin{tikzpicture}[baseline=-.5pt, scale=0.4]
    \draw (0,0) rectangle (1,1);
    \draw (0,1) -- (1,0);
    \node[scale=0.6] at (0.7, 0.7) {$s$};
    \node[scale=0.6] at (0.3, 0.3) {$k$};
  \end{tikzpicture}%
}

\newcommand{\norsq}{%
  \begin{tikzpicture}[baseline=-.4pt, scale=0.4]
    \draw (0,0) rectangle (1,1);
    \node[scale=0.75] at (0.5, 0.5) {$r$};
  \end{tikzpicture}%
}

\begin{document}

\title[Signed counting and combinatorial proofs]{Signed Counting on Restricted Partitions and Combinatorial Proofs of Three Identities}

\author[S. Fu]{Shishuo Fu}
\address[Shishuo Fu]{College of Mathematics and Statistics \& Center for Discrete Mathematics, Chongqing University, Chongqing 401331, China}
\email{fsshuo@cqu.edu.cn}
\author[C. Wang]{Chenwei Wang$^{\ast}$}
\address[Chenwei Wang]{College of Mathematics and Statistics, Chongqing University, Chongqing 401331, China}
\email{20211357@stu.cqu.edu.cn}
\thanks{$^{\ast}$Corresponding author.}

\date{\today}

	\begin{abstract}
	Signed enumerations of $\ell$-regular partitions by the parity of the number of parts are known to correspond to partitions with congruence conditions. Inspired by the recent combinatorial approaches of Ballantine–Merca and Liu for such identities, we provide combinatorial proofs of three identities for $\ell$-regular partitions and their variants. Two of these were originally established analytically by Hickerson and Robbins, respectively.
	\end{abstract}
	\maketitle % 生成标题、作者和日期信息
    \noindent \textbf{Keywords:} Partition identity; $\ell$-regular partition; signed counting; combinatorial proof.
	%\keywords{Partition identity; $\ell$-regular partition; signed counting; combinatorial proof}
		% \newline \indent 2020 {\it Mathematics Subject Classification}. 11P84, 05A17, 05A15, 33B10.
	%%%%%%%%%%%%%%%%%%%%%%%%%%%%%%%%%%%%%%%%%%%%%%%%%%%%%%%%%%%%%%%%%%
	\section{Introduction}
	A \textit{partition} $\lambda$ of a positive integer $n$ is a finite nonincreasing sequence of positive integers $(\lambda_1, \lambda_2, \ldots, \lambda_r)$ such that $\sum_{i=1}^r \lambda_i = n$. As usual, we call $\lambda_i$ a part of $\lambda$, let $|\lambda|$ denote the \textit{size} (or \emph{weight}) of $\lambda$, i.e. $|\lambda| =\sum_{i=1}^r \lambda_i$, and denote $\len(\lambda)$ the \textit{length} of $\lambda$, i.e. $\len(\lambda) = r$. We denote by $p(n)$ the number of partitions of $n$ and we have the generating function
	\begin{equation}
		\sum_{n=0}^{\infty} p(n)q^n = \frac{1}{(q;q)_\infty} ,
	\end{equation}
	here and throughout, we use the following $q$-series notation \cite[Section 2.2]{andrews1998}: for $|q|<1$,
	\begin{align*}
		&(a;q)_{\infty}=(1-a)(1-aq)(1-aq^{2})\cdots ;\\
		&(a;q)_{0}=1;\\
		&(a;q)_{n}=\frac{(a;q)_{\infty}}{(aq^{n};q)_{\infty}}=(1-a)(1-aq)(1-aq^{2})\cdots(1-aq^{n-1}) ,\quad \text{for some } n > 0.
	\end{align*}
	
	Suppose $S$ is a certain subclass of partitions. For notational convenience and uniformity, we make the convention that when superscript ``$e$'' (resp. ``$o$'') is added as $S^{e}$ (resp.~$S^{o}$), it refers to the subset of $S$ containing partitions that have an even (resp. odd) number of parts. The parameter $n$ in the parentheses indicates that we restrict the set to partitions of $n$. Moreover, whenever an uppercase letter is used for a set, the corresponding lowercase letter refers to its cardinality, and we let
	$$s^{\pm}:=s^{e}-s^{o}.$$
	We refer to this as \emph{the signed counting with respect to the length} over the set $S$.
	
	A partition is called \textit{distinct} if none of its parts repeats. Let $D(n)$ denote the set of distinct partitions of $n$, and $d(n)$ be the number of distinct partitions of $n$. The generating function for $d(n)$ is given by
	\begin{equation}
		\sum_{n=0}^{\infty} d(n)q^n=(-q;q)_{\infty}.
	\end{equation}
	
	The celebrated \textit{Euler's Pentagonal Number Theorem} \cite[Theorem 1.6]{andrews1998} can be stated as a signed counting result on the set of distinct partitions. More precisely, let $D^e(n)$ (resp. $D^o(n)$) be the set of distinct partitions of $n$ into an even (resp. odd) number of parts and let $d^e(n)$ (resp. $d^o(n)$) be its cardinality.
	We have
	\begin{equation}\label{eq:PNT}
		(q;q)_{\infty}=\sum_{n\geq0}(d^e(n)-d^o(n))q^n=\sum_{k=-\infty}^{+\infty}(-1)^{k}q^{\omega(k)},
	\end{equation}
	where $\omega(k):=\frac{(3k-1)k}{2}$ is the generalized pentagonal number, and we use the convention $d^e(0)=1$, $d^o(0)=0$, i.e., the empty partition $\varnothing$ is thought of as having zero parts. This theorem was first established via generating functions by Euler, and Franklin \cite[Theorem~1.6]{andrews1998} provided a combinatorial proof of it which will be recalled in Section \ref{EPT}.
	
	Additionally, there is the famous \textit{Euler's Partition Theorem} \cite[Corollary 1.2]{andrews1998}:
	\begin{equation}
		\frac{1}{(q;q^2)}_{\infty}=(-q;q)_{\infty}.
		\label{Euler}
	\end{equation}
	
	For an integer $\ell > 1$, a partition is called \textit{$\ell$-regular} if none of its parts is divisible by $\ell$. Let $B_{\ell}(n)$ denote the set of $\ell$-regular partitions of $n$. The generating function for $b_{\ell}(n)=|B_{\ell}(n)|$ is given by~\cite[Corollary 1.3]{andrews1998}
	\begin{equation}\label{gf:l-reg}
		\sum_{n=0}^{\infty} b_{\ell}(n)q^n=\frac{(q^{\ell};q^{\ell})_\infty}{(q;q)_\infty}.
	\end{equation}
	
	The signed counting on $B_{\ell}(n)$ have been investigated recently by Liu~\cite[Theorems 1.1 and 1.2]{liu2025} via combinatorial approach, his results split into two cases as follows. When ${\ell}>1$ is an even integer, for any positive integer $n$, we have 
	\begin{equation}
		b^{\pm}_{\ell}(n):=b^e_{\ell}(n)-b^o_{\ell}(n)=(-1)^n d_{\ell}(n),
		\label{eq1}
	\end{equation}
	where $d_{\ell}(n)$ denotes the number of partitions of $n$ into distinct parts which are {\bf either odd or divisible} by $\ell$. When ${\ell}>1$ is an odd integer, for any positive integer $n$, we have
	\begin{equation}
		b^{\pm}_{\ell}(n)=(-1)^n c_{\ell}(n),
		\label{eq2}
	\end{equation}
	where $c_{\ell}(n)$ denotes the number of partitions of $n$ into distinct parts which are {\bf odd and not divisible} by $\ell$.
	
	It is worth noting that when $\ell=2$, \eqref{eq1} is equivalent to Euler's theorem \eqref{Euler}. Moreover, Ballantine and Merca \cite{ballantine20234,ballantine20236} gave combinatorial proofs for the ${\ell}=4,6$ cases of \eqref{eq1}.
	
	This work focuses on combinatorial proofs for several identities involving the signed counting of restricted partitions related to $\ell$-regular partitions. We introduce, for any integers $0\le m < \ell$, the more general partition set $B_{{\ell},m}(n)$, which is the set of partitions of $n$ into parts not congruent to $m$ modulo ${\ell}$ (note that $B_{\ell,0}(n)=B_{\ell}(n)$). Then it is natural to consider the signed counting on this set. Denote by $D_{\ell,m}(n)$ the set of partitions of $n$ into distinct parts which are either odd or congruent to $m$ modulo $\ell$. We have the following signed counting identity dealing with the case $\ell\equiv m\equiv 0\pmod 2$, which could be viewed as our first main result. 
	\begin{Theorem}
		Let ${\ell}>m\ge 0$ be two even integers. For each positive integer $n$, we have
		\begin{equation}
			b_{{\ell},m}^{\pm}(n)=(-1)^nd_{{\ell},m}(n).
			\label{eq3}
		\end{equation}
		\label{thm1}
	\end{Theorem}
	\begin{remark}\label{rmk:special case}
		We note that the $(\ell,m)=(4,2)$ case was recently established in two ways by Ballantine and Merca \cite[Theorem~1.1(i)]{ballantine20234}, while the $m=0$ case reduces to \eqref{eq1}.
	\end{remark}
	
	The following equivalent form of the right hand side of \eqref{gf:l-reg} 
	$$\frac{(q^{\ell};q^{\ell})_{\infty}}{(q;q)_{\infty}}=\prod_{k\ge 1}(1+q^k+q^{2k}+\cdots+q^{(\ell-1)k})$$
	connects $B_{\ell}(n)$ with $Q_{\ell-1}(n)$\footnote{In representation theory, the term ``$\ell$-regular partition'' refers to $Q_{\ell-1}$; see for instance \cite[p.~251]{JK1981}}, which denotes the set of partitions of $n$ with each part occurring no more than $\ell-1$ times (note that $Q_1(n)=D(n)$). Our next two results are concerned with the signed counting on $Q_r(n)$ for $r>1$. Firstly for the case of even $r$, Hickerson \cite{hickerson1973} showed that, for any positive integer $n$, 
	\begin{equation}
		q^{\pm}_{r}(n)=(-1)^{n}c_{r+1}(n),
		\label{delta}
	\end{equation}
	where $c_{r+1}(n)$ is the same function introduced in \eqref{eq2}. Applying a classical bijection between $Q_{r}(n)$ and $B_{r+1}(n)$ due to Glaisher \cite{glaisher1883} (see Theorem~\ref{thm:Glaisher} in Section~\ref{Gla}), Liu \cite[Theorem 1.4]{liu2025} provided a combinatorial proof of \eqref{delta}.

	In contrast, the signed counting on $Q_r(n)$ for odd $r$ appears to be less straightforward and we have found two versions in the literature, for each of which we shall provide a combinatorial proof featuring sign-reversing involutions.
	
	\begin{Theorem}[\textbf{Hickerson~\cite{hickerson1978}}]
		If $s,t,u,n$ are positive integers with $s$ odd and $1\leq s<t$, then
		\begin{equation}
			q^{\pm}_{2tu-1}(n)=(-1)^{n}\sum_{j}f_{s,t,u}(n-tj^2-(t-s)j),
			\label{eq5}
		\end{equation}
		here $f_{s,t,u}(n)$ is the number of partitions of $n$ in which each odd part occurs at most once and is $\not\equiv \pm s \tpmod{2t}$, and each even part is divisible by $2t$ and occurs less than $u$ times.
		\label{thm2}
	\end{Theorem}
	
	\begin{Theorem}[\textbf{Robbins~\cite{robbins2002}}]
		Let $r\geq2$, for any positive integer $n$, we have
		\begin{equation}
			q^{\pm}_{2r-1}(n)=\sum_{k=-\infty}^{\infty}(-1)^{k}b_{r}{(\frac{n-\omega(k)}{2})},
			\label{eq6}
		\end{equation}
		where $\omega(k)=\frac{k(3k-1)}{2}$.
		\label{thm3}
	\end{Theorem}
	
	\begin{remark}
		In 1973, Hickerson \cite{hickerson1973} showed a formula for $q^{\pm}_{3}(n)$, then in 1975, Alder and Muwafi \cite{alder1975} found formulae for $q^{\pm}_{5}(n)$ and $q^{\pm}_{7}(n)$. Finally in 1978, Hickerson \cite{hickerson1978} obtained the general formula~\eqref{eq5} for any positive odd integer $r$. We remark that Robbins's formula~\eqref{eq6} is much recent and looks simpler than \eqref{eq5}. In contrast, the right hand side of Hickerson's formula~\eqref{eq5} has a uniform sign relying only on the parity of $n$, and thus is cancellation free.
	\end{remark}
	
	 A key fact that was needed in Liu's combinatorial proof of \eqref{delta} is that Glaisher's bijection will not change the parity of the number of parts when $r$ is even. But this property no longer holds for odd $r$, suggesting that Liu's combinatorial approach to \eqref{delta} may not extend to this case. New ideas need to be introduced in our combinatorial proofs of \eqref{eq5} and \eqref{eq6}.

The rest of our paper is organized as follows. We shall recall three bijections in Section \ref{Pre}. The combinatorial proofs of Thoerems \ref{thm1}, \ref{thm2}, and \ref{thm3} are given in Sections \ref{sec:proof of Thm1}, \ref{Po2}, and \ref{Po3}, respectively. For ease of reference, Table~\ref{tab:Symbols} in the appendix summarizes the notation for the various classes of restricted partitions.

\section{Preliminaries\label{Pre}}
We recall in this section three mappings (one bijection and two involutions) that will play key roles in our proofs of the main theorems. 

\subsection{Glaisher's Bijection from \texorpdfstring{$Q_{r}(n)$ to $B_{r+1}(n)$}{Qr(n) to Br+1(n)}\label{Gla}}
For every $r\ge 1$, Glaisher constructed in 1883 a famous bijection, say $\tau_r$, to show that
\begin{Theorem}[{\cite{glaisher1883}}]\label{thm:Glaisher}
	For any positive integer $n$, we have
	\begin{align*}
		q_{r}(n) = b_{r+1}(n).
	\end{align*}
\end{Theorem}
The mapping $\tau_r$ and its inverse $\tau^{-1}_r$ are briefly described as follows.
\begin{itemize}
	\item $\tau_r: Q_{r}(n) \to B_{r+1}(n)$. If the partition $\lambda\in Q_r(n)$ contains a part, say $\lambda_i$, that is divisible by $r+1$, then we split $\lambda_i$ into $r+1$ parts of size $\lambda_i/(r+1)$. We repeat this procedure until no parts are divisible by $r+1$, and take this final partition to be the image $\tau_r(\lambda)$.
	\item $\tau^{-1}_r: B_{r+1}(n)\to Q_{r}(n)$. If the partition $\mu\in B_{r+1}(n)$ contains a part, say $\mu_i$, that occurs at least $r+1$ times, then we merge $r+1$ copies of $\mu_i$ into a single part of size $(r+1)\mu_i$. We repeat this procedure until all part sizes occur at most $r$ times, and take this final partition to be the preimage $\tau^{-1}(\mu)$.
\end{itemize}

Obviously, the celebrated Euler's ``odd vs. distinct'' theorem corresponds to the case with $r = 1$. Glaisher's bijection $\tau_r$ will be utilized in the proofs of Theorems~\ref{thm1} and \ref{thm3}.

\subsection{Franklin's Involution for Euler's Pentagonal Number Theorem\label{EPT}}
To give a purely combinatorial proof of Euler's Pentagonal Number Theorem \eqref{eq:PNT}, Franklin devised a clever operation on the \emph{Ferrers diagram} associated with every distinct partition. This celebrated involution will be used in Section \ref{Po3} for the proof of Theorem \ref{thm3}. Here the Ferrers diagram refers to a commonly used graphical representation of partitions. We use left-justified and horizontally placed $i$ squares to represent the part of size $i$, arranged from top to bottom. For instance, Example~\ref{542} shows the Ferrers diagram of $\lambda=(5,4,2)$.

For any distinct partition $\lambda=(\lambda_1,\lambda_2,\ldots,\lambda_{\ell})$, we denote $b:=\lambda_{\ell}$ and call it the \emph{bottom} of $\lambda$, and we also let $s$, called the \emph{slope} of $\lambda$, be the largest index $1\le i\le \ell$, such that $\lambda_i=\lambda_1-i+1$. Geometrically, for the Ferrers diagram associated with $\lambda$, the `bottom' and `slope' are given by the lengths of the last row and the rightmost anti-diagonal, respectively.

\begin{Example}\label{542} $\lambda=(5,4,2)$, $b=2$, $s=2$. 
	
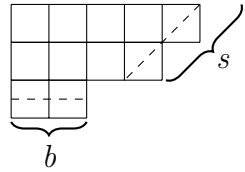
\begin{center}
		\begin{tikzpicture}[scale=0.5]
			\draw (0,0) grid (2,1);
			\draw (0,1) grid (4,2);
			\draw (0,2) grid (5,3);
			\draw[dashed] (5,3) -- (3,1);
			\draw[dashed] (0,0.5) -- (2,0.5);
			\draw[decorate, decoration={brace, amplitude=5pt, mirror}, thick] (4.1, 1) -- (6.1, 3);
			\node at (5.6,1.5) {s};
			\draw[decorate, decoration={brace, amplitude=5pt, mirror}, thick] (0, -0.1) -- (2, -0.1);
			\node at (1,-1) {b};
		\end{tikzpicture}
	\end{center}
\end{Example}

Given a distinct partition $\lambda$ with bottom $b$ and slope $s$, Franklin's involution $\psi$ is defined with respect to the comparison between $b$ and $s$. There are three cases.
\begin{enumerate}
	\item[Case~1.] If $b=s$ or $b=s+1$, and bottom and slope have a common square like the two diagrams in Example~\ref{eg:b-s-cross}, we keep it unchanged.
	\item[Case~2.] For $\lambda$ not belonging to {Case 1}, if $b>s$, we remove the slope and attach a row of the same length $s$ below the original bottom to form a new bottom. 
	\item[Case~3.] For $\lambda$ not belonging to {Case 1}, if $b\leq s$, we remove the bottom and attach diagonally a line of length $b$ to the right of the original slope to form a new slope.
\end{enumerate}

	We can easily observe that $\psi$ is an involution that fixes the partitions in Case~1, and interchanges the partitions in Case~2 and Case~3. Furthermore, we list without proof the following facts about the involution $\psi$.
	\begin{enumerate}
		\item Denote by $\Omega$ the set of all partitions in Case~1 (i.e., the set of fixed points under $\psi$). The size of each $\lambda\in\Omega$ must be a pentagonal number $\omega(k)=\frac{k(3k-1)}{2}$ and $\len(\lambda)\equiv k \pmod{2}$ for a certain $k\in\bZ$; see Example~\ref{eg:b-s-cross} below.
		\item Conversely, every pentagonal number corresponds to a unique distinct partition that is fixed by $\psi$.
		\item The preimage $\lambda$ and the image $\psi(\lambda)$ has opposite parity for their lengths.
	\end{enumerate}

\begin{Example} \label{eg:b-s-cross}
$\lambda=(5,4,3)$, $b=3$, $s=3$ and $\lambda=(6,5,4)$, $b=4$, $s=3$.
	
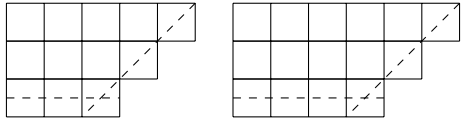
\begin{center}
				\begin{tikzpicture}[scale=0.5]
					\draw (0,0) grid (3,1);
					\draw (0,1) grid (4,2);
					\draw (0,2) grid (5,3);
					\draw[dashed] (5,3) -- (2,0);
					\draw[dashed] (0,0.5) -- (3,0.5);
					\draw (6,0) grid (10,1);
					\draw (6,1) grid (11,2);
					\draw (6,2) grid (12,3);
					\draw[dashed] (12,3) -- (9,0);
					\draw[dashed] (6,0.5) -- (10,0.5);
				\end{tikzpicture}	
		\end{center}
\end{Example}

\subsection{The Kolitsch-Kolitsch Involution for Jacobi's Triple Product Identity\label{JTP}}
The following involution $\phi_{r,s}$ was due to Kolitsch and Kolitsch~\cite{kolitsch2018}, who used it to give a combinatorial proof of the Jacobi's triple product identity~\cite[Theorem~2.8]{andrews1998} in the form 
\begin{equation}
	(zq^s;q^r)_{\infty}(z^{-1}q^{r-s};q^r)_{\infty}(q^r;q^r)_{\infty}=\sum_{i=-\infty}^{\infty}(-z)^i q^{\frac{r(i^2 -i)}{2}+si},
	\label{eqjtp}
\end{equation}
where $r>s$ are both positive integers. We need it in our proof of Theorem \ref{thm2}.

For the remainder of this subsection, we will make frequent use of a special diagram called \emph{$r$-modular (Ferrers) diagram}, in which we weight each square by $r$ (instead of $1$ in standard Ferrers diagram), and allow upper triangles weighted by $s$ and lower triangles weighted by $k:=r-s$, respectively; see the diagrams in Example~\ref{eg:Kolitsch} for an illustration.

To make the paper self-contained, we briefly recall the involution $\phi_{r,s}$ due to Kolitsch-Kolitsch below. The reader should refer to the original paper~\cite{kolitsch2018} for further information. One key ingredient in Kolitsch-Kolitsch's approach was certain ``paste and cut'' procedure over $r$-modular diagrams, such an operation was reminiscent of Wright's combinatorial proof~\cite{wright1965} (see also \cite[Section~6.2.1]{pak2006}) of Jacobi's triple product identity in another equivalent form. 

Let $\lambda$ be a partition into distinct parts congruent to $0$ or $\pm s $ modulo $r$. It takes three main steps to construct the image $\phi_{r,s}(\lambda)$. First, we split $\lambda$ into three subpartitions: $\lambda_{0}$, $\lambda_{+}$, and $\lambda_{-}$ that consist of parts congruent to $0$, $s$, and $-s$ modulo $r$, respectively. In brief, both $\lambda_{+}$ and $\lambda_{-}$ are represented as properly indented $r$-modular diagrams. These two diagrams are then aligned by their smallest parts to be merged into a single $r$-modular diagram, say $\mu$. The reader is referred to the following example for an illustration. $\lambda_0$ is represented as an $r$-modular diagram as well.

	\begin{Example}\label{eg:Kolitsch}
	The $r$-modular diagram generated from $\lambda_{+}$ and $\lambda_{-}$ in two cases: $\len(\lambda_{+})\ge \len(\lambda_{-})$ or $\len(\lambda_{+})<\len(\lambda_{-})$. 
		\begin{center}
			\begin{tikzpicture} [scale=0.6]			
				\draw (0,5) -- (5,0);	
				\foreach \x in {1,...,5}
				\draw (5,\x) -- (5-\x,\x);
				\foreach \x in {1,...,5}
				\draw (\x,5) -- (\x,5-\x);
				\foreach \x in {0,...,3}
				\foreach \y in {0,...,\x}
				\node at (\x+1.5,5-\y-0.5) {$r$};
				\foreach \x in {0,...,4}
				\node at (\x+0.75,4.75-\x){$s$};
				\draw (5,5) grid (6,2);
				\draw (6,5) grid (7,3);
				\foreach \x in {0,1}
				\foreach \y in {\x,...,2}
				\node at (5.5+\x,2.5+\y){$r$};
				\draw (2,3) grid (5,0);
				\draw (2,0) grid (3,-1);		
				\foreach \x in {1,2,3}
				\node at (2.5,\x-1.5) {$r$};
				\node at (3.5,0.5) {$r$};
				\foreach \x in {1,2,3}
				\node at (\x+1.35,3.3-\x) {$k$};

				\draw (8,5) -- (13,0);
				
				\foreach \x in {1,...,5}
				\draw (\x+8,5-\x) -- (0+8,5-\x);
				\foreach \x in {0,...,4}
				\draw (\x+8,0) -- (\x+8,5-\x);
				
				\foreach \x in {0,...,3}
				\foreach \y in {0,...,\x}
				\node at (2-\x+1.5+8,\y+0.5) {$r$};
				
				\foreach \x in {0,...,4}
				\node at (\x+0.35+8,4.3-\x){$k$};
				
				\draw (3+8,2) grid (5+8,0);
				\draw (1+8,3) grid (6+8,2);
				
				\draw (5+8,0) grid (4+8,-1);
				\draw (2+8,0) grid (0+8,-3);
				\draw (4+8,0) grid (2+8,-2);

				\foreach \x in {3,...,5}
				\node at (\x+0.5+8,2.5) {$r$};
				\node at (4.5+8,1.5) {$r$};
				\foreach \x in {0,...,2}
				\node at (\x+2.75+8,2.75-\x){$s$};
				\foreach \x in {0,...,4}
				\node at (\x+0.5+8,-0.5) {$r$};
				\foreach \x in {0,...,3}
				\node at (\x+0.5+8,-1.5) {$r$};
				\foreach \x in {0,1}
				\node at (\x+0.5+8,-2.5) {$r$};
			\end{tikzpicture}	
		\end{center}
	\end{Example}

Second, in the merged $r$-modular diagram $\mu$, we replace the square \bisq~concatenated from $s$ and $k=r-s$ by the normal square \norsq, preserving the total weight $r$. There are two cases for the ensuing operations.

\begin{enumerate}
	\item[Case 1.] $\len(\lambda_{+})\geq \len(\lambda_{-})$; see the left diagram in Example~\ref{eg:Kolitsch}. We denote by $l$ the length of the longest column in $\mu$ that consists of only squares, and we set $l:=0$ when there are no such columns in $\mu$. Denote by $m$ the length of the first row of $\lambda_0$ (as an $r$-modular diagram) and set $m=0$ if $\lambda_0=\varnothing$. There are three subcases to consider.
	\begin{enumerate}
		\item[(1a)] If $l=m=0$, we simply set $\phi_{r,s}(\lambda)=\lambda$.
		\item[(1b)] If $l>m$, we remove a column of length $l$ from $\mu$ and append it as a new first row to $\lambda_{0}$. The diagrams thus obtained are denoted as $\hat{\mu}$ and $\hat{\lambda}_0$, respectively.
		\item[(1c)] If $m\ge l$ and $m>0$, we remove the first row from $\lambda_0$ and insert a column of length $m$ into $\mu$. Insert it to the far right in the case of $l=0$. The diagrams thus obtained are denoted as $\hat{\mu}$ and $\hat{\lambda}_0$, respectively.
	\end{enumerate}
	\item[Case 2.] $\len(\lambda_{+}) < \len(\lambda_{-})$; see the right diagram in Example~\ref{eg:Kolitsch}. We transpose the merged diagram $\mu$, perform the same operations as Case 1, then transpose it back to get $\hat{\mu}$. 
\end{enumerate}

Third, draw the diagonal line in $\hat{\mu}$ to split it into a partition with parts congruent to $s$ modulo $r$ denoted as $\hat{\lambda}_{+s}$, and a partition with parts congruent to $k\equiv -s$ modulo $r$ denoted as $\hat{\lambda}_{-s}$. Combine the triple $(\hat{\lambda}_0,\hat{\lambda}_{+s},\hat{\lambda}_{-s})$ into a single partition, which is taken to be $\phi_{r,s}(\lambda)$.

For our later use, we calculate the size of the nonempty partitions that are fixed by $\phi_{r,s}$, whose $r$-modular diagrams are depicted below.
	\begin{center}
		\begin{tikzpicture} [scale=0.6]
			\draw (0,5) -- (5,0);
			
			\foreach \x in {1,...,5}
			\draw (5,\x) -- (5-\x,\x);
			\foreach \x in {1,...,5}
			\draw (\x,5) -- (\x,5-\x);

			\foreach \x in {0,...,3}
			\foreach \y in {0,...,\x}
			\node at (\x+1.5,5-\y-0.5) {$r$};
			
			\foreach \x in {0,...,4}
			\node at (\x+0.75,4.75-\x){$s$};
			\draw (6,5) -- (11,0);
			
			\foreach \x in {1,...,5}
			\draw (6+\x,5-\x) -- (6,5-\x);
			\foreach \x in {0,...,4}
			\draw (\x+6,0) -- (6+\x,5-\x);
			
			\foreach \x in {0,...,3}
			\foreach \y in {0,...,\x}
			\node at (8-\x+1.5,\y+0.5) {$r$};
			
			\foreach \x in {0,...,4}
			\node at (\x+6.35,4.3-\x){$k$};
		\end{tikzpicture}	
	\end{center}
The first case corresponds to $\lambda_{-}=\lambda_0=\varnothing$ and $\lambda_{+}$ being a staircase. We denote this $i$-staircase (for some $i>0$) as
	\begin{align}
		\triangle_i(r,s) &:= (s+(i-1)r,s+(i-2)r,\ldots,s+r,s).\label{eq:i-stair>0}
	\end{align}
The second case corresponds to $\lambda_{+}=\lambda_0=\varnothing$ and $\lambda_{-}$ being a staircase. We denote this $i$-staircase (for some $i<0$) as
	\begin{align}
		\triangle_i(r,s) &:= (-s-ir,-s-(i+1)r,\ldots,-s+r).\label{eq:i-stair<0}
	\end{align}
In both cases, it can be computed that $|\triangle_i(r,s)|=\frac{r(i^2 -i)}{2}+si$. The $0$-staircase is simply the empty partition, i.e., $\triangle_0(r,s)=\varnothing$.

\section{Two proofs of Theorem \ref{thm1}}\label{sec:proof of Thm1}
We are going to provide in this section two proofs of Theorem~\ref{thm1}, one analytic and one combinatorial. Since $B_{\ell,\ell}=B_{\ell,0}$ and $D_{\ell,\ell}=D_{\ell,0}$, for notational convenience, we consider the cases $\ell \ge m > 0$ instead of $\ell>m\ge 0$, where $\ell=2\ell'$ and $m=2m'$ are two fixed even integers.

\subsection{An analytic proof of Theorem~\ref{thm1}}

Although the generating function proof of \eqref{eq3} is quite routine, we include it here for the sake of completeness. 
\begin{proof}[1st proof of Theorem~\ref{thm1}]
On the one hand, according to the definition of $D_{\ell,m}$, we see that
\begin{align*}
	\sum_{n\ge 0}d_{\ell,m}(n)q^n &= (-q^m;q^{\ell})_{\infty}(-q;q^2)_{\infty}.
\end{align*}
Replacing $q$ by $-q$, we derive that
\begin{align}\label{eq:thm1-rhs}
\sum_{n\ge 0}(-1)^n d_{\ell,m}(n)q^n &= (-q^m;q^{\ell})_{\infty}(q;q^2)_{\infty}.
\end{align}
On the other hand, the definition of $B_{\ell,m}$ directly implies that
\begin{align*}
	\sum_{\lambda\in B_{\ell,m}}z^{\len(\lambda)}q^{|\lambda|} &= \frac{1}{\prod\limits_{\substack{i=1,\, i\not = m}}^{\ell} (zq^i;q^\ell)_\infty}=\frac{(zq^{m};q^{\ell})_{\infty}}{\prod_{i=1}^{\ell} (zq^i;q^\ell)_\infty}=\frac{(zq^{m};q^{\ell})_{\infty}}{(zq;q)_\infty}.
\end{align*}
Plugging in $z=-1$, we deduce that
\begin{align*}
\sum_{n\ge 0}b^{\pm}_{\ell,m}(n)q^n &= \frac{(-q^{m};q^{\ell})_{\infty}}{(-q;q)_\infty}=(-q^m;q^{\ell})_{\infty}(q;q^2)_{\infty},
\end{align*}
which agrees with the right-hand side of \eqref{eq:thm1-rhs}, as desired.
\end{proof}

\subsection{A combinatorial proof of Theorem~\ref{thm1}}
In this subsection, we give a combinatorial proof of Theorem~\ref{thm1} for the cases $\ell>m>0$, since the case $\ell=m$ corresponds to the identity~\eqref{eq1}, which already has a combinatorial proof by Liu~\cite[Theorem~1.1]{liu2025}. We prefer the frequency notation $\lambda=1^{\f(1)}2^{\f(2)}\cdots$ for a given partition $\lambda$, where $\f(i)$ denotes the number of times that part $i$ occurs in $\lambda$. We associate with a given partition $\lambda\in B_{\ell,m}$ three statistics defined as follows:
\begin{align}
	\alpha_{\ell,m}(\lambda) &:=\max\{\lambda_{i}\mid \lambda_i+1\equiv \f(\lambda_{i})\equiv 1\tpmod{2}\},\\
	\beta_{\ell,m}^{(1)}(\lambda) &:=\max\{\lambda_{i}\mid \f(\lambda_{i})>1,\,\lambda_{i}\not\equiv m'\tpmod{\ell'}\},\\
	\beta_{\ell,m}^{(2)}(\lambda) &:=\max\{\lambda_{i}\mid \f(\lambda_{i})>3,\,\lambda_{i}\equiv m'\tpmod{\ell'}\}.
\end{align}
We make the convention that $\max(S)=0$ when $S=\varnothing$. In what follows, we use the abbreviations $\alpha(\lambda)=\alpha_{\ell,m}(\lambda)$, $\beta^{(1)}(\lambda)=\beta_{\ell,m}^{(1)}(\lambda)$, and $\beta^{(2)}(\lambda)=\beta_{\ell,m}^{(2)}(\lambda)$. When the partition $\lambda$ itself is clear from the context, we even suppress the symbol $\lambda$ to write $\alpha$, $\beta^{(1)}$, and $\beta^{(2)}$.
\begin{remark}\label{ab0}
	Note that for any $\lambda\in B_{\ell,m}$, if $\beta^{(1)}(\lambda)$ and $\beta^{(2)}(\lambda)$ are not both zero, then $\beta^{(1)}(\lambda)\not=2\beta^{(2)}(\lambda)$, since $2\beta^{(2)}(\lambda)\equiv m\pmod{\ell}$, and therefore it cannot be a part of $\lambda$. Furthermore, if $\alpha(\lambda)=\beta^{(1)}(\lambda)=\beta^{(2)}(\lambda)=0$, then $\lambda$ is a partition such that each even part is congruent to $m'$ modulo $\ell'$ and occurring exactly twice, while each odd part is either distinct, or congruent to $m'$ modulo $\ell'$ and occurring twice or thrice.
\end{remark}

Now we present a combinatorial proof of Theorem \ref{thm1} (cases $\ell>m>0$) by composing an involution $\psi_{\ell,m}$ with a bijection $\tau_{\ell,m}$. A concrete example can be found at the end of this subsection.

\begin{proof}[2nd proof of Theorem~\ref{thm1}]	
We first let
\begin{align*}
	B^{(0)}:=\{\lambda\in B_{\ell,m} \mid \alpha=\beta^{(1)}=\beta^{(2)}=0\},
\end{align*}
then further split $B_{\ell,m}\setminus B^{(0)}$ into four disjoint subsets 
$$B_{\ell,m} \setminus B^{(0)}= B^{(1)}\cup B^{(2)}\cup B^{(3)}\cup B^{(4)},$$ 
basing on the relative magnitudes of $\alpha$, $2\beta^{(1)}$, and $4\beta^{(2)}$, where
\begin{align*}
		B^{(1)}&:=\left\{\lambda\mid2\beta^{(1)}>\max\{4\beta^{(2)},\alpha\}\right\}, \\
		B^{(2)}&:=\left\{\lambda\mid 4\beta^{(2)}>\max\{2\beta^{(1)},\alpha\}\right\}, \\
		B^{(3)}&:=\left\{\lambda\mid \alpha\geq \max\{2\beta^{(1)},4\beta^{(2)}\},\, \alpha/2\not\equiv m\tpmod{\ell}\right\}, \\
		B^{(4)}&:=\left\{\lambda\mid \alpha\ge \max\{2\beta^{(1)},4\beta^{(2)}\},\, \alpha/2\equiv m\tpmod{\ell}\right\}. 
\end{align*}
	In view of Remark~\ref{ab0}, the four subsets above are indeed disjoint and cover the set $B_{\ell,m}\setminus B^{(0)}$. Our involution $\psi_{\ell,m}:B_{\ell,m}\to B_{\ell,m}$ is defined as follows. We explain case-by-case how to get the image $\mu:=\psi_{\ell,m}(\lambda)$ for a given $\lambda\in B_{\ell,m}$.
\begin{enumerate}
	\item If $\lambda\in B^{(0)}$, then $\mu:=\lambda$, so the set $B^{0}$ are the fixed points under $\psi_{\ell,m}$.
	\item If $\lambda\in B^{(1)}$, then we merge two parts of size $\beta^{(1)}$ into a single part $2\beta^{(1)}$ and let $\mu$ be the partition thus obtained. For $\lambda\in B^{(1)}$, we have $\beta^{(1)}(\lambda)\not\equiv m'\tpmod{\ell'}$, thus $2\beta^{(1)}(\lambda)\not\equiv m\tpmod{\ell}$, and $2\beta^{(1)}(\lambda)>\max\{4\beta^{(2)}(\lambda),\alpha(\lambda)\}$. Consequently, $\mu\in B_{\ell,m}$ and the part size $2\beta^{(1)}(\lambda)$ occurs one more time in $\mu$ than in $\lambda$. More precisely, since $2\beta^{(1)}(\lambda)>\alpha(\lambda)$, we see that $\f(2\beta^{(1)}(\lambda))$ must be even for $\lambda$ and be odd for $\mu$, implying that
	\begin{align*}
		\alpha(\mu) &=2\beta^{(1)}(\lambda)\geq2\beta^{(1)}(\mu),\\
		\alpha(\mu) &=2\beta^{(1)}(\lambda)>4\beta^{(2)}(\lambda)=4\beta^{(2)}(\mu),\text{ and }\\
		\alpha(\mu)/2 &=\beta^{(1)}(\lambda)\not\equiv m'\tpmod{\ell'}.
	\end{align*}
	These three conditions ensure that $\mu\in B^{(3)}$.
	\item If $\lambda\in B^{(2)}$, then we merge four parts of size $\beta^{(2)}$ into a single part $4\beta^{(2)}$ and let $\mu$ be the partition thus obtained. For $\lambda\in B^{(2)}$, we have $\beta^{(2)}(\lambda)\equiv m'\tpmod{\ell'}$, thus $2\beta^{(2)}(\lambda)\equiv m\tpmod{\ell}$, $4\beta^{(2)}(\lambda)\equiv 2m \not\equiv m\tpmod{\ell}$, and $4\beta^{(2)}(\lambda)>\max\{2\beta^{(1)}(\lambda),\alpha(\lambda)\}$. Consequently, $\mu\in B_{\ell,m}$ and $\f(4\beta^{(2)}(\lambda))$ must be even for $\lambda$ and odd for $\mu$, implying that
	\begin{align*}
		\alpha(\mu) &=4\beta^{(2)}(\lambda)\geq4\beta^{(2)}(\mu),\\
		\alpha(\mu) &=4\beta^{(2)}(\lambda)>2\beta^{(1)}(\lambda)=2\beta^{(1)}(\mu),\\
		\alpha(\mu)/2 &=2\beta^{(2)}(\lambda) \equiv m\tpmod{\ell}.
	\end{align*}
	These three conditions ensure that $\mu\in B^{(4)}$.
	\item If $\lambda\in B^{(3)}$, then we split one part $\alpha$ into two copies of $\alpha/2$ and let $\mu$ be the partition thus obtained. Note that by definition, $\frac{\alpha}{2}\not\equiv m \tpmod{\ell}$, so $\mu \in B_{\ell,m}$. Moreover, $\alpha\not\equiv m\tpmod{\ell}$, thus $\alpha/2\not\equiv m'\tpmod{\ell'}$. Similar arguments as in (2) and (3) lead to the following constraints:
	\begin{align*}
		2\beta^{(1)}(\mu) &=\alpha(\lambda)\ge 4\beta^{(2)}(\lambda)\ge 4\beta^{(2)}(\mu),\\
		2\beta^{(1)}(\mu) &=\alpha(\lambda)>\alpha(\mu).
	\end{align*}
	In the first string of inequalities, the equality $2\beta^{(1)}(\mu)=4\beta^{(2)}(\mu)$ cannot be achieved owing to Remark~\ref{ab0}. Therefore we see that $\mu\in B^{(1)}$.
	\item If $\lambda\in B^{(4)}$, then we split one part $\alpha$ into four copies of $\alpha/4$ and let $\mu$ be the partition thus obtained. Note that $\alpha(\lambda)/2\equiv m\tpmod{\ell}$ thus $\alpha(\lambda)/4\equiv m'\tpmod{\ell'}$, and $\lambda\in B^{(4)}$ implies that $\alpha(\lambda)/4\ge \beta^{(2)}(\lambda)$. Similar arguments as in previous cases give rise to the inequalities:
	\begin{align*}
		4\beta^{(2)}(\mu) &=\alpha(\lambda)\ge 2\beta^{(1)}(\lambda)\ge 2\beta^{(1)}(\mu),\\
		4\beta^{(2)}(\mu) &=\alpha(\lambda)>\alpha(\mu).
	\end{align*}
	Applying again Remark~\ref{ab0} to exclude the equality $4\beta^{(2)}(\mu)=2\beta^{(1)}(\mu)$, we deduce that $\mu\in B^{(2)}$.
\end{enumerate}
Evidently, Cases (2) and (4) (as well as Cases (3) and (5)) feature inverse operations. In summary of all five cases, we conclude that $\psi_{\ell,m}$ is a well-defined, weight-preserving involution over $B_{\ell,m}$, such that $\psi_{\ell,m}(B^{(0)})=B^{(0)}$, and for $i=1,2,3,4$, we have
	\begin{align*}
		\psi_{\ell,m}(B^{(i)}) &= B^{(i+2)},
	\end{align*} 
where the addition in the superindex is modulo $4$, like $\psi_{\ell,m}(B^{(3)})=B^{(5)}=B^{(1)}$, etc. And if $\lambda\in B_{\ell,m}\setminus B^{(0)}$ and $\mu:=\psi_{\ell,m}(\lambda)$, then
\begin{align}\label{psi:len parity}
	\len(\mu) \equiv \len(\lambda)+1 \pmod 2.
\end{align}

By the characterization of the set of fixed points $B^{(0)}$ given in Remark~\ref{ab0}, we know that for every $\lambda\in B^{(0)}$,  
\begin{align}\label{psi:weight-len}
	|\lambda|\equiv \len(\lambda) \pmod 2.
\end{align}
In view of \eqref{psi:len parity} and \eqref{psi:weight-len}, to prove \eqref{eq3} it suffices to show that 
\begin{align}\label{eq:B0=dlm}
	|B^{(0)}|=d_{\ell,m}. 
\end{align}
To this end, we construct a weight-preserving bijection $\tau_{\ell,m}:B^{(0)}\to D_{\ell,m}$ as follows, which is a variant of Glaisher's bijection recalled in Section~\ref{Gla}. 

For any $\lambda\in B^{(0)}$ and a part of it, say $\lambda_i$, we rephrase the characterization given by Remark~\ref{ab0} in terms of the frequency $\f(\lambda_i)$: 
\begin{enumerate}
	\item If $\f(\lambda_i)=1$ then $\lambda_i$ must be odd.
	\item If $\f(\lambda_i)>1$ then $\lambda_i\equiv m'\tpmod{\ell'}$, and either $\lambda_i$ is even with $\f(\lambda_i)=2$, or $\lambda_i$ is odd with $\f(\lambda_i)=2$ or $3$.
\end{enumerate}
Now to get its image $\xi:=\tau_{\ell,m}(\lambda)$, we replace two identical parts $\lambda_i$ in $\lambda$ by a single part $2\lambda_i$, and repeat this process until all parts are distinct. The final partition is taken as $\xi$. Note that $\lambda_i$ is repeated in $\lambda$ precisely when it belongs to case (2) above, whence $2\lambda_i\equiv m\tpmod{\ell}$. Therefore, we see that $\xi\in D_{\ell,m}$ so $\tau_{\ell,m}$ is well-defined.

Conversely, for every $\xi \in D_{\ell,m}$, we can find a unique $\lambda \in B^{(0)}$ such that $\tau_{\ell,m}(\lambda)=\xi$, by splitting every part in $\xi$ that is congruent to $m$ modulo $\ell$ into two equal parts. We trust the reader to verify that the partition thus obtained does belong to $B^{(0)}$. 

Hence, $\tau_{\ell,m}$ is indeed a bijection so we have \eqref{eq:B0=dlm} and the proof is now complete.
\end{proof}
\begin{Example}
	For $\ell=4$, $m=2$, $n=8$, we see
	\begin{align*}
		B_{4,2}(8) &=
		\left\{ 
		8^{1}, 1^{1}7^{1}, 3^{1}5^{1}, 1^{3}5^{1}, 4^{2}, 1^{1}3^{1}4^{1}, 1^{4}4^{1}, 
		1^{2}3^{2}, 1^{5}3^{1}, 1^{8}
		\right\}, \\
		D_{4,2}(8) &=\{1^{1}7^{1}, 2^{1}6^{1}, 3^{1}5^{1}, 1^{1}2^{1}5^{1}\}.
	\end{align*}
	The pairing inside $B_{4,2}(8)$ is given by $\psi_{4,2}$, which explains the cancellation $b^e_{4,2}(8)-b^o_{4,2}(8)$
	\begin{align*}
		4^{2} &\overset{\psi_{4,2}}{\longleftrightarrow} 8^{1},\\
		1^{5}3^{1} &\longleftrightarrow 1^{1}3^{1}4^{1},\\
		1^{8} &\longleftrightarrow 1^{4}4^{1},
	\end{align*}
	and the fixed points $B^{(0)}=\{1^{1}7^{1}, 3^{1}5^{1}, 1^{3}5^{1}, 1^{2}3^{2}\}$. Moreover, the correspondences according to $\tau_{4,2}$ are
	\begin{align*}
		1^{1}7^{1} &\overset{\tau_{4,2}}{\longrightarrow} 1^{1}7^{1},\\
		3^{1}5^{1} &\longrightarrow 3^{1}5^{1},\\
		1^{3}5^{1} &\longrightarrow 1^{1}2^{1}5^{1},\\
		1^{2}3^{2} &\longrightarrow 2^{1}6^{1}.
	\end{align*}
\end{Example}
%\begin{remark}
%	We observe that, with respect to the definitions of the three statistics introduced so far, when $m=\ell$, we see that $\beta_{\ell,\ell}^{(2)}(\lambda)=0$, and the involution $\psi_{\ell,\ell}$ is essentially equivalent to that of Liu~\cite[Section 2]{kolitsch2018}. When $l=4$, $m=0$, our involution differs from, but is essentially equivalent to, the involution of Ballantine and Merca \cite[Section 2.2]{ballantine20234}; one can also verify that it differs from that involution when $l=4$, $m=2$. When $l=6$, $m=0$, our involution is equivalent to the involution of Ballantine and Merca \cite[Section 3.2]{ballantine20236}.
%\end{remark}
\begin{remark}
 For the special case $(\ell,m)=(4,2)$,	our involution $\psi_{4,2}$ differs from Ballantine and Merca's $\varphi$~\cite[Theorem~1.1(i)]{ballantine20234}. Indeed, for the partition $\lambda=1^8$, one sees that $\psi_{4,2}(\lambda)=1^4 4$ while $\varphi(\lambda)=8$.
\end{remark}
%\textbf{Remark.} Actually, in solving this problem, we didn't first determine $\psi_{l,m}$ and then $\tau_{l,m}$. Instead, we first established a proof structure combining an involution and a bijection, and constructed $\tau^{-1}_{l,m}$ which is easy to conceive if familiar with Glaisher's classic bijection.

\section{Proof of Theorem \ref{thm2}\label{Po2}}
In this section, $s$, $t$ and $u$ are positive integers with $s$ being odd and $1\leq s<t$. We first recall a useful operation for making a bigger partition from smaller ones. If $\lambda$ and $\mu$ are two partitions, then the \emph{union} of $\lambda$ and $\mu$, denoted $\lambda\cup\mu$, refers to the partition obtained by taking the multiset union of their parts and arranging the parts in nonincreasing order. For example, if $\lambda=(4,2,1)$ and $\mu=(3,2,2)$, then $\lambda\cup\mu=(4,3,2,2,2,1)$. It is clear from the definition that
\begin{align*}
	\len(\lambda \cup \mu)=\len(\lambda)+\len(\mu).
\end{align*}

Among the three combinatorial proofs constructed in this paper, the proof of Theorem~\ref{thm2} given in this section is the most complicated one and it essentially parallels Hickerson's original analytic proof. Denote by $F_{s,t,u}$ the set of partitions in which each odd part occurs at most once and is $\not\equiv \pm s \tpmod{2t}$, and each even part is divisible by $2t$ and occurs less than $u$ times. We begin with a lemma that takes care of one crucial step in the proof of Theorem~\ref{thm2}.

\begin{Lemma}\label{lem:lambda1}
For positive integers $s$, $t$, and $u$ with $s$ being odd and $1\leq s<t$, the signed counting with respect to the length is the same for $Q_{2t-1}$ and $J_{s,t}$, i.e.,
\begin{align}
|Q^e_{2t-1}|-|Q^o_{2t-1}|=|J^e_{s,t}|-|J^o_{s,t}|.\label{Q and J}
\end{align}
Here $J_{s,t}$ is defined via the $i$-staircase introduced in \eqref{eq:i-stair>0} and \eqref{eq:i-stair<0} as
\begin{align}
		J_{s,t} &= \{\nu \mid \nu=\triangle_i(2t,s)\cup \nu^* \text{ for some $i\in\bZ$ }\},\label{def:Jst}
\end{align}
where $\nu^*\in D$ is a subpartition of $\nu$ whose parts are all odd and $\not\equiv \pm s\tpmod{2t}$.
\end{Lemma}

\begin{proof}
	Take any partition $\lambda\in Q_{2t-1}$, we decompose it as
	\begin{align*}
		\lambda=\lambda^{(e)}\cup\lambda^{(o)},
	\end{align*}
	where for a generic part $\lambda^{(e)}_i$ (resp., $\lambda^{(o)}_i$) in $\lambda^{(e)}$ (resp., $\lambda^{(o)}$), we require that
	\begin{align*}
		&\f(\lambda^{(e)}_i) \equiv 0 \pmod{2}, \text{ and }\f(\lambda^{(e)}_i) \leq 2t-2,\\
		&\f(\lambda^{(o)}_i) =1.
	\end{align*}
	We also introduce
	\begin{align}
		\mo_{t}(\lambda) &:= \max\{\lambda^{(o)}_i\mid\lambda^{(o)}_i/2= \lambda^{(e)}_j\in \bZ,\text{ and }0\le \f(\lambda^{(e)}_j)<2t-2\},\label{e2} \\
		\me_{t}(\lambda) &:= \max\{\lambda^{(e)}_i\mid 2\lambda^{(e)}_i\not\in \lambda^{(o)} \},\label{l2}
	\end{align}
	with the same convention that $\max(S)=0$ when $S=\varnothing$. Specifically, \eqref{e2} implies that $\mo_t(\lambda)=0$ if and only if each even part in $\lambda^{(o)}$ has its half occurring precisely $2t-2$ times in $\lambda^{(e)}$, while \eqref{l2} indicates that $\me_t(\lambda)=0$ if and only if each part in $\lambda^{(e)}$ has its double being a part of $\lambda^{(o)}$. 

  Our journey from $Q_{2t-1}$ to $J_{s,t}$ consists of three steps: $\psi_t$, $\sigma_t$, and $\hat{\phi}_{s,t}$, as illustrated below.
  \begin{align*}
  Q_{2t-1} \xrightarrow[\text{fixed points}]{\text{$\psi_t$}} \hat{Q}_{2t-1} \xrightarrow[]{\text{$\sigma_t$}} D_{2t,0} \xrightarrow[\text{fixed points}]{\text{$\hat{\phi}_{s,t}$}} J_{s,t}.
  \end{align*}

	\begin{enumerate}
		\item[\textbf{Step 1:}] We begin with the involution $\psi_{t}\colon Q_{2t-1} \to Q_{2t-1}$. It is a Franklin-like involution that depends on the two values $\mo_t(\lambda)$ and $2\me_t(\lambda)$. For any $\lambda$ in $Q_{2t-1}$, its image $\mu:=\psi_t(\lambda)$ is constructed according to the following three cases.
		\begin{enumerate}
			\item[Case 1.] If $\mo_t(\lambda)=\me_t(\lambda)=0$, then $\lambda$ is fixed by $\psi_t$, i.e., we let $\mu=\lambda$ and denote by $\hat{Q}_{2t-1}$ the set of all such fixed partitions. For the next two cases, $\mo_t(\lambda)$ and $\me_t(\lambda)$ cannot both be zero.

			\item[Case 2.] If $\mo_t(\lambda)> 2\me_t(\lambda)$, we split the part $\mo_t(\lambda)$ into two copies of $\mo_t(\lambda)/2$, and let $\mu$ be the partition thus obtained. Note that $\mu\in Q_{2t-1}$ and $\mo_t(\mu)<\mo_t(\lambda)=2\me_t(\mu)$.
			
			\item[Case 3.] If $\mo_t(\lambda) <2\me_t(\lambda)$, we merge two copies of $\me_t(\lambda)$ into a single part $2\me_t(\lambda)$, and let $\mu$ be the partition thus obtained. Note that $\mu\in Q_{2t-1}$ and $\mo_t(\mu)=2\me_t(\lambda)>2\me_t(\mu)$.
		\end{enumerate}

	From this construction, we see that $\psi_{t}$ is a well-defined involution over $Q_{2t-1}$ that fixes $\hat{Q}_{2t-1}$ and that $\lambda$ belongs to Case 2 if and only if $\psi_t(\lambda)$ is in Case 3. 
		
	\item[\textbf{Step 2:}] We use the discussion after \eqref{l2} to give an explicit characterization of the set $\hat{Q}_{2t-1}$, i.e., the set of partitions fixed by $\psi_{t}$. A partition $\lambda=\lambda^{(e)}\cup\lambda^{(o)}\in \hat{Q}_{2t-1}$, if and only if it satisfies the following conditions.
	\begin{enumerate}
		\item For every part $a$ in $\lambda^{(e)}$, it occurs precisely $2t-2$ times in $\lambda^{(e)}$ and $2a$ must be a part of $\lambda^{(o)}$.
		\item For every even part $b$ in $\lambda^{(o)}$, $b/2$ must be a part of $\lambda^{(e)}$.
	\end{enumerate}
	We next define a bijection $\sigma_{t}:\hat{Q}_{2t-1}\to D_{2t,0}$, where recall that
		\begin{align*}
			D_{2t,0}:=
			\{\mu\in D \mid \text{if $a\in \mu$ is even, then $a \equiv 0\tpmod{2t}$}\}.
		\end{align*}
	Keeping in mind the two conditions (a) and (b) satisfied by any partition $\lambda=\lambda^{(e)}\cup\lambda^{(o)}\in\hat{Q}_{2t-1}$, the image $\mu:=\sigma_t(\lambda)$ is constructed as follows. For each part $a\in\lambda^{(e)}$, we sum up the $(2t-2)$ copies of $a$ in $\lambda^{(e)}$ and the single part $2a$ in $\lambda^{(o)}$ to get a part of size $2ta$, leaving all the odd parts of $\lambda^{(o)}$ unchanged. Let this new partition be $\mu$, then clearly $\mu\in D_{2t,0}$.
	
  Conversely, for any $\mu\in D_{2t,0}$, we split every even part, say $2ta$, in $\mu$ into $2t-2$ copies of part $a$ and one copy of part $2a$. These new parts, together with the existing odd parts of $\mu$, form a partition denoted as $\lambda$. It is not hard to check that $\lambda\in\hat{Q}_{2t-1}$ and $\sigma_t(\lambda)=\mu$. Hence $\sigma_{t}$ is a bijection as claimed. Moreover, since $\sigma_t$ keeps turning $(2t-2)+1=2t-1$ parts into one part, we observe that
  \begin{align}\label{sigma-len}
  \len(\lambda)\equiv \len(\sigma_t(\lambda))\tpmod 2.
  \end{align}
		
		\item[\textbf{Step 3:}]
		Next, we decompose $\mu \in D_{2t,0}$ into four subpartitions:
		\begin{equation}
			\mu = (\mu_0\cup \mu_{+} \cup \mu_{-}) \cup \mu^*=:\mu^{\dagger} \cup \mu^*,
		\end{equation}
		where $\mu_{+}$, $\mu_{-}$, and $\mu_0$ consist of all parts in $\mu$ congruent to $s$, $-s$, and $0$ modulo $2t$, respectively, while $\mu^*$ is the subpartition that contains all of the remaining parts (necessarily they are all odd parts since $\mu\in D_{2t,0}$). The involution $\hat{\phi}_{s,t}$ on $D_{2t,0}$ is induced by the involution $\phi_{2t,s}$ (for the case $r=2t$) introduced in Section~\ref{JTP}.
		\begin{align*}
			\hat{\phi}_{s,t}: D_{2t,0} &\to D_{2t,0} \\
			\mu^{\dagger} \cup \mu^* &\mapsto \phi_{2t,s}(\mu^{\dagger}) \cup \mu^*.
		\end{align*}
		Consequently, $\mu$ is fixed by $\hat{\phi}_{s,t}$ if and only if $\mu^{\dagger}$ is fixed by $\phi_{2t,s}$, giving rise to the set of fixed points $J_{s,t}$ as defined in \eqref{def:Jst}. Finally, we apply \eqref{sigma-len} to deduce that
		$$\len(\lambda) \equiv \len(\hat{\phi}_{s,t}\circ\sigma_t\circ\psi_t(\lambda))\tpmod 2,$$
		which readily implies \eqref{Q and J}.
\end{enumerate}
\end{proof}

\begin{proof}[Proof of Theorem~\ref{thm2}]
The proof proceeds in three main steps. 
\begin{enumerate}[I.]
	\item For every partition $\lambda\in Q_{2tu-1}$, weighted by $(-1)^{\len(\lambda)}$, we rewrite it as the union of two subpartitions $\lambda=\lambda^{(1)}\cup\lambda^{(2)}$, where $\lambda^{(1)}\in Q_{2t-1}$ and $\lambda^{(2)}$ is a subpartition of $\lambda$ with each part occurring a multiple of $2t$ times and no more than $2t(u-1)$ times. By the \emph{Division Algorithm}, we know that such kind of decomposition always exists and is unique. Using the notation from the previous section, this is equivalent to saying that
\begin{align}
	\f(\lambda^{(1)}_i) & <2t,\nonumber \\
	\f(\lambda^{(2)}_i) & \equiv 0 \tpmod{2t}, \text{ and } \f(\lambda^{(2)}_i)<2tu, \label{lambda2}
\end{align}
where $\lambda^{(1)}_i$ (resp., $\lambda^{(2)}_i$) is a part of $\lambda^{(1)}$ (resp., $\lambda^{(2)}$). Next, we deal with $\lambda^{(1)}$ and $\lambda^{(2)}$ separately. For each $\lambda^{(2)}$, we merge every $2t$ identical parts into a single part to get a partition, say $\tilde{\lambda}^{(2)}$. Then the constraints \eqref{lambda2} ensure that $\tilde{\lambda}^{(2)}\in Q_{u-1}$, and every part of it is divisible by $2t$.
	\item Note that $\lambda^{(1)}\in Q_{2t-1}$, so by Lemma~\ref{lem:lambda1}, either it is paired with another partition with the same weight but opposite sign, or it survives both involutions $\psi_t$ and $\hat{\phi}_{s,t}$, so that $\nu:=\hat{\phi}_{s,t}\circ\sigma_t\circ\psi_t(\lambda^{(1)})$ belongs to $J_{s,t}$, allowing us to write
	\begin{align*}
	\nu &= \triangle_i(2t,s)\cup \nu^*,
	\end{align*}
	where $i\in\bZ$ and $\nu^*\in D$ contains only odd parts that are $\not\equiv \pm s\tpmod{2t}$.
	\item Note that the union $\tilde{\lambda}^{(2)}\cup \nu^*$ is a partition in $F_{s,t,u}$. Conversely, every partition from $F_{s,t,u}$ can be uniquely written as such a union, all it takes is to separate even parts from odd parts.
\end{enumerate}
In summary of the above steps, we see that beginning with a partition $\lambda=\lambda^{(1)}\cup\lambda^{(2)}$ from $Q_{2tu-1}$, on the one hand $\lambda^{(2)}$ is mapped to $\tilde{\lambda}^{(2)}$. On the other hand, if $\lambda^{(1)}$ survives both involutions, i.e.,
$$\hat{\phi}_{s,t}\circ\sigma_t\circ\psi_t(\lambda^{(1)}) = \sigma_t(\lambda^{(1)})=:\nu,$$
then $\nu=\triangle_i(2t,s)\cup \nu^*\in J_{s,t}$, and 
$$\tilde{\lambda}^{(2)}\cup \nu^*=:\eta \in F_{s,t,u}.$$
Furthermore, tracking the weights, we see that
\begin{align}
|\lambda| &= |\lambda^{(1)}|+|\lambda^{(2)}| = |\nu|+|\tilde{\lambda}^{(2)}|= |\triangle_i(2t,s)|+|\nu^*|+|\tilde{\lambda}^{(2)}|=t(i^2-i)+si+|\eta|.\label{eq:weight}
\end{align}
While the sign transforms as
\begin{align}
(-1)^{\len(\lambda)} &= (-1)^{\len(\lambda^{(1)})} \stackrel{*}{=} (-1)^{\len(\nu)} \stackrel{\dagger}{=} (-1)^{|\nu|} = (-1)^{|\lambda^{(1)}|} = (-1)^{|\lambda|}.\label{eq:length}
\end{align}
The equality marked $*$ follows from \eqref{sigma-len}, while the one marked $\dagger$ uses the fact that $\nu$ has only odd parts, so $\len(\nu)\equiv |\nu|\tpmod 2$.

Combining \eqref{eq:weight} with \eqref{eq:length}, we deduce \eqref{eq5} and the proof is now complete.
\end{proof}

\begin{Example}
	For $s=1$, $t=u=2$, $n=9$, note that all partitions in $F_{1,2,2}$ are distinct partitions into multiples of $4$. More precisely, we see that
	\begin{align*}
	& Q_{7}(9)=\left\{
		\begin{aligned}
			& 9^{1}, 1^{1}8^{1}, 2^{1}7^{1}, 1^{2}7^{1}, 3^{1}6^{1}, 1^{1}2^{1}6^{1}, 1^{3}6^{1}, 4^{1}5^{1}, 1^{1}3^{1}5^{1}, 2^{2}5^{1}, 1^{2}2^{1}5^{1}, 1^{4}5^{1}, 1^{1}4^{2}, \\
			& 2^{1}3^{1}4^{1}, 1^{2}3^{1}4^{1}, 1^{1}2^{2}4^{1},1^{3}2^{1}4^{1}, 1^{5}4^{1}, 3^{3}, 1^{1}2^{1}3^{2}, 1^{3}3^{2}, 2^{3}3^{1},1^{2}2^{2}3^{1}, 1^{4}2^{1}3^{1}, 1^{6}3^{1}, \\
			& 1^{1}2^{4},1^{3}2^{3}, 1^{5}2^{2}, 1^{7}2^{1}
		\end{aligned}
		\right\},\\
	&\bigcup_{j=-\infty}^{\infty}F_{1,2,2}(9-2j^2-j)=F_{1,2,2}(9)\cup F_{1,2,2}(6)\cup F_{1,2,2}(8)\cup F_{1,2,2}(3) = \{8^{1}\}.
	\end{align*}
	
	We have the pairings inside $Q_3$ given by $\psi_{2}$,
	\begin{align*}
		1^{1}8^{1} &\overset{\psi_{2}}{\longleftrightarrow} 1^{1}4^{2}, &\quad 2^{1}3^{1}4^{1} &\overset{\psi_{2}}\longleftrightarrow 2^{3}3^{1},\\
		2^{1}7^{1} &\longleftrightarrow 1^{2}7^{1},&\quad 1^{2}3^{1}4^{1} &\longleftrightarrow 1^{2}2^{2}3^{1},\\
		3^{1}6^{1} &\longleftrightarrow 3^{3}, &\quad 1^{3}2^{1}4^{1} &\longleftrightarrow 1^{3}2^{3},\\
		1^{1}2^{1}6^{1} &\longleftrightarrow 1^{1}2^{1}3^{2}, &\quad 1^{1}4^{1} &\longleftrightarrow 1^{1}2^{2},\\
		1^{3}6^{1} &\longleftrightarrow 1^{3}3^{2}, &\quad 2^{1}3^{1} &\longleftrightarrow 1^{2}3^{1},\\
		4^{1}5^{1} &\longleftrightarrow 2^{2}5^{1}.
	\end{align*}
	as well as the fixed points $\hat{Q}_{3}=
	\left\{
	9^{1}, 1^{1}3^{1}5^{1}, 1^{2}2^{1}5^{1}, 5^{1}, 1^{1}2^{2}4^{1}, 1^{1}, 1^{3}2^{1}
	\right\}$.
	
	Next, the correspondences according to $\sigma_{2}:\hat{Q}_3\to D_{4,0}$ are
	\begin{align*}
		9^{1} &\overset{\sigma_{2}}{\longrightarrow} 9^{1}, & 1^{1}2^{2}4^{1} &\overset{\sigma_{2}}\longrightarrow 1^{1}8^{1},\\
		1^{1}3^{1}5^{1} &\longrightarrow 1^{1}3^{1}5^{1}, &1^{1} &\longrightarrow 1^{1},\\
		1^{2}2^{1}5^{1} &\longrightarrow 4^{1}5^{1},  &1^{3}2^{1} &\longrightarrow 1^{1}4^{1}.\\
		5^{1} &\longrightarrow 5^{1},
	\end{align*}
	so that $D_{4,0} =
	\left\{
	9^{1}, 1^{1}3^{1}5^{1}, 4^{1}5^{1}, 5^{1}, 1^{1}8^{1}, 1^{1}, 1^{1}4^{1}
	\right\}
	$.
	
	Then we have the correspondences inside $D_{4,0}$
	given by $\hat{\phi}_{1,2}$
	\begin{align*}
		9^{1} &\overset{\hat{\phi}_{1,2}}{\longrightarrow} 4^{1}5^{1},\\
		1^{1}3^{1}5^{1} &\longrightarrow 1^{1}8^{1},\\
		5^{1} &\longrightarrow 1^{1}4^{1},
	\end{align*}
	and the fixed points $J_{1,2}=\{1^{1}\}$. This partition traces back to $1^1 2^4\in Q_7(9)$, which has the decomposition $1^1 2^4 = \lambda^{(1)}\cup\lambda^{(2)}= 1^1 \cup 2^4$. In turn we have $\tilde{\lambda}^{(2)}=8^1$, $\hat{\phi}_{1,2}\circ\sigma_2\circ\psi_2(1^1)=1^1=\triangle_1(4,1)\cup\varnothing$, and hence $8^1\cup\varnothing=8^1\in F_{1,2,2}(8)$.
	
	% Finally, the correspondence according to $\tau_{1,2,2}$ is
	% \begin{equation*}
	% 	1^{1}2^{4} \overset{\tau_{1,2,2}}{\longrightarrow} 8^{1}.
	% \end{equation*}
\end{Example}

\section{Proof of Theorem \ref{thm3} \label{Po3}}
In this section, we fix $r\geq 2$ and provide a combinatorial proof of Theorem \ref{thm3}. 
\begin{proof}[Proof of Theorem~\ref{thm3}]
	For any $\lambda$ in $Q_{2r-1}$, we can uniquely express $\lambda$ as the union of three subpartitions: 
	\begin{equation*}
		\lambda = \lambda^{(1)} \cup \lambda^{(2)}\cup\lambda^{(2)},
	\end{equation*}
	where $\lambda^{(1)}\in D$ and $\lambda^{(2)}\in Q_{r-1}$. As a result, Franklin's involution $\psi$ over $D$ naturally induces an involution over $Q_{2r-1}$:
	\begin{align*}
	\hat{\psi}_r: Q_{2r-1} &\to Q_{2r-1} \\
	\lambda^{(1)} \cup \lambda^{(2)}\cup\lambda^{(2)} &\mapsto \psi(\lambda^{(1)}) \cup \lambda^{(2)}\cup\lambda^{(2)},
	\end{align*}
	so that $\lambda$ is fixed by $\hat{\psi}_r$ if and only if $\lambda^{(1)}$ is fixed by $\psi$.

Now we apply Facts (1)--(3) for Franklin's involution $\psi$ given in Section~\ref{EPT} to derive that
\begin{align*}
\len(\hat{\psi}_r(\lambda)) \equiv \begin{cases}
\len(\lambda)+1 \tpmod 2 & \text{if $\hat{\psi}_r(\lambda)\neq \lambda$,} \\
\len(\lambda)\equiv k \tpmod 2 & \text{if $\hat{\psi}_r(\lambda)= \lambda$ and $|\lambda^{(1)}|=\omega(k)$ for some $k\in \bZ$.}
\end{cases}
\end{align*}
Evidently, the first case contributes nothing to the signed counting $q_{2r-1}^{\pm}$, while the second case produces
\begin{align}
q^{\pm}_{2r-1}(n) &= \sum_{k=-\infty}^{\infty}(-1)^k q_{r-1}(\frac{n-\omega(k)}{2}).\label{eq:q_2r-1=q_r-1}
\end{align}
Finally, we deduce \eqref{eq6} from \eqref{eq:q_2r-1=q_r-1} by applying Glaisher's bijection for Theorem~\ref{thm:Glaisher}.

\end{proof}

%%%%%%%%%%%%%%%%%%%%%%%%%%%%%%%%%%%%%%%%%%%%%%%%%%%
\section*{Acknowledgement}
Shishuo Fu was supported by the National Natural Science Foundation of China grant 12171059 and the Fundamental Research Funds for the Central Universities (grant no.~2025CDJ-IAISYB-008).

%We are grateful to the anonymous referees for making useful suggestions on revising this paper. 
\newpage 
\appendix

\section{Notation}
Here we will list most of the symbols appearing in this paper, along with their meanings, in the order of their occurrences in our paper.
% \begin{center}
% 	\textbf{List of Symbols}
% \end{center}

	\begin{table}[H]
			\centering  % 表格居中
	
	\begin{tabular}{|p{4cm}<{\centering} |p{11cm}|}  % 列格式（|表示竖线，c居中）
		
		\hline
		Notation & Definition \\
		\hline
		$S$ & The set of partitions that satisfy certain conditions.   \\
		\hline
		$s$ & The number of partitions in $S$.\\
		\hline
		$S^e$(resp. $S^o$) & The subset of partitions whose number of parts is even(resp. odd).\\
		\hline
		$s^e$(resp. $s^o$) & The number of partitions in $S^e$(resp. $S^o$).\\
		\hline  
		$s^{\pm}$ & $s^{\pm}:=s^e-s^o$\\
		\hline
		$D$ & The set of distinct partitions. \\
		\hline
		%			$d(n)$ & The number of distinct partitions of $n$. \\
		%			\hline
		%			$D^e(n)$(resp. $D^o(n)$) & The set of distinct partitions of $n$ with an even(resp. odd) number of parts.	\\
		%			\hline
		%			$d^e(n)$(resp. $d^o(n)$) & The number of distinct partitions of $n$ with an even(resp. odd) number of parts. \\
		%			\hline
		$\omega(k)$ & Pentagonal numbers $\omega(k)=\frac{(3k-1)k}{2}.$ \\
		\hline
		$B_{\ell}$ & The set of $\ell$-regular partitions. \\
		\hline
		%			$b_{\ell}(n)$ & The number of $\ell$-regular partitions of $n$. \\
		%			\hline
		%			$B^e_{\ell}(n)$(resp. $B^o_{\ell}(n)$) & The set of $\ell$-regular partitions of $n$ into an even(resp. odd) number of parts. \\
		%			\hline
		%			$b^e_{\ell}(n)$(resp. $b^o_{\ell}(n)$) & The number of $\ell$-regular partitions of $n$ into an even(resp. odd) number of parts. \\
		%			\hline
		$d_{\ell}$ & The number of distinct partitions with each part either odd or divisible by $\ell$. \\
		\hline
		$c_{\ell}$ & The number of distinct partitions with each part odd and not divisible by $\ell$. \\
		\hline
		$B_{\ell,m}$ & The set of partitions with each part not congruent to $m$ modulo $\ell$. \\
		\hline
		%			$b_{\ell,m}(n)$ & The number of partitions of $n$ with each part not congruent to $m$ modulo $\ell$. \\
		%			\hline
		%			$B_{\ell,m}^e(n)$(resp. $B_{\ell,m}^o(n)$) & The set of partitions of $n$ into an even(resp. odd) number of parts which are not congruent to $m$ modulo $\ell$. \\
		%			\hline
		%			$b_{\ell,m}^e(n)$(resp. $b_{\ell,m}^o(n)$) & The number of partitions of $n$ into an even(resp. odd) number of parts which are not congruent to $m$ modulo $\ell$. \\
		%			\hline
		$D_{\ell,m}$ & The set of distinct partitions with each part either odd or congruent to $m$ modulo $\ell$. \\
		\hline
		%			$d_{\ell,m}(n)$ & The number of partitions of $n$ with each part either odd or congruent to $m$ modulo $\ell$. \\
		%			\hline
		$Q_{i}$ & The set of partitions with each part occurring at most $i$ times \\
		\hline
		%			$q_{i}(n)$ & The number of partitions of $n$ with each part occurring at most $i$ times \\
		%			\hline
		%			$Q_{i}^e(n)$(resp. $Q_{i}^o(n)$) & The set of partitions of $n$ into an even(resp. odd) number of parts with each part occurring at most $i$ times \\
		%			\hline
		%			$q_{i}^e(n)$(resp. $q_{i}^o(n)$) & The number of partitions of $n$ into an even(resp. odd) number of parts with each part occurring at most $i$ times. \\
		%			\hline
		%			$q^{\pm}_{r}(n)$ & $q^{\pm}_{r}(n)=q_{r}^e(n)-q_{r}^o(n)$. \\
		%			\hline
		$F_{s,t,u}$ & The set of partitions in which each odd part occurs at most once and is $\not\equiv \pm s \pmod{2t}$ and in which each even part is divisible by $2t$ and occurs less than $u$ times. \\
		\hline
		%			$f_{s,t,u}(n)$ & The number of partitions of $n$ in which each odd part occurs at most once and is $\not\equiv \pm s \pmod{2t}$ and in which each part is divisible by $2t$ and occurs less than $u$ times. \\
		%			\hline
	\end{tabular}
	\smallskip
	\caption{A list of symbols}\label{tab:Symbols}  % 交叉引用标签
	\end{table}

\end{document}